\documentclass[letterpaper, 10 pt, conference]{ieeeconf}
\IEEEoverridecommandlockouts                              
\usepackage{amsthm}
\usepackage{enumerate}
\usepackage{amsfonts} 
\usepackage{amsmath} 
\usepackage{algorithm}
\newcommand{\N}{\mathcal{N}}
\usepackage{algorithmic}
\usepackage{amssymb}
\usepackage{xcolor}

\usepackage{url}
\usepackage{lipsum,adjustbox}
\usepackage{tikz}
\usepackage{multirow}
\let\labelindent\relax
\usepackage[shortlabels]{enumitem}
\renewcommand{\theenumi}{(\roman{enumi})}

\usepackage{cite}
\newtheorem{theorem}{Theorem}[section]
\newtheorem{proposition}[theorem]{Proposition}
\newtheorem{lemma}[theorem]{Lemma}
\newtheorem{corollary}[theorem]{Corollary}
\theoremstyle{remark}

\theoremstyle{definition}

\newtheorem{definition}{Definition}

\DeclareMathOperator{\Lc}{\mathcal{L}}

\DeclareMathOperator{\real}{\mathbb{R}}
\DeclareMathOperator{\softp}{\operatorname{softplus}}
\DeclareMathOperator{\proj}{\operatorname{proj}_{\Omega_m}}
\DeclareMathOperator{\projx}{\operatorname{proj}_{\Omega_n}}
\newcommand{\longthmtitle}[1]{\mbox{} \textit{(#1):}}

\renewcommand{\baselinestretch}{1.025}

\title{Accelerated Algorithms for a Class of Optimization Problems with Constraints
\thanks{The first author was supported by a Siemens Fellowship. This work was supported in part by Boeing Strategic University Initiative.}}

\author{Anjali Parashar \quad  Priyank Srivastava \quad Anuradha M. Annaswamy
\thanks{A. Parashar, P. Srivastava and A. M. Annaswamy are with the Department of Mechanical Engineering, Massachusetts Institue of Technology.
Email:\texttt{\{anjalip,psrivast,aanna\}@mit.edu}.
}
}

\begin{document}
\maketitle

\begin{abstract}
    This paper presents a framework to 
    solve constrained optimization problems in an accelerated manner
    based on High-Order Tuners (HT).
    Our approach is based on reformulating the original constrained problem 
    as the unconstrained optimization of a loss function.
    We start with convex optimization problems and identify the conditions under which the loss function is convex.
    Building on the insight that the loss function could be convex even if the original optimization problem is not,
    we extend our approach to a class of nonconvex optimization problems.
    The use of a HT together with this approach enables us to achieve a convergence rate better than state-of-the-art gradient-based methods.
    Moreover, for equality-constrained optimization problems,
    the proposed method ensures that the state remains feasible throughout the evolution,
    regardless of the convexity of the original problem.
\end{abstract}
\thispagestyle{empty}

\section{Introduction}
Several problems in a wide range of fields such as adaptive control, machine learning, and optimization  \cite{Narendra_1989, SB-LV:09, hazan2019introduction, Nesterov2018LecturesOptimization} utilize gradient-descent based approaches. In adaptive control, decision making in the form of reducing a performance error as well ensuring that the learning error in parameters reduces to zero  as well as learning the  unknown parameters of the system, both in realtime are the underlying goals. A gradient-descent approach is often employed to realize both goals, first to obtain a fast convergence of the performance error and then, to reduce the learning error. In machine learning, fast and correct training of models such as neural networks is sought after, which necessitates the reduction of an underlying loss function using a gradient-based approach.  Optimization approaches require the solution of an augmented Lagrangian in an expedient manner, through a gradient-descent method. All of these solutions have to be achieved often in the presence of various constraints. Conservation equations invariably introduces equality constraints in the form of momentum, mass, or energy balance, while capacity limits introduce inequality constraints. Given the importance of the fast convergence in all these problems, there is a need for algorithms that can lead to an order of magnitude improvement in the speed of convergence, both performance and learning errors, while retaining stability. This paper proposes such an algorithm.

Recently, a class of High-order Tuners (HT) was proposed in continuous-time \cite{Gaudio2020ASystems} and in discrete-time~\cite{Gaudio2020AcceleratedRegressors} for a large class of dynamic systems for the purpose of estimation of unknown parameters.  The estimation problem for this class of systems can be reformulated as a linear regression problem, where the underlying regressors correspond to various system variables that can be measured, including inputs, outputs, and states. Each of these high-order tuners was shown to result in a stable performance error when the regressors were time-varying. One of these tuners was extended in \cite{boffi2020implicit} for a class of nonlinear problems where the underlying error model is still based on linear regression. The motivation for these methods come from adaptive control perspectives in \cite{morse1992high,evesque2003adaptive} and optimization perspectives \cite{YEN:83,Nesterov2018LecturesOptimization}.
In~\cite{Jose2021-LCSS}, these results were extended to convex functions, and shown that the high-order tuners are capable of leading to stable performance with a potential for accelerated convergence of the performance error to zero. 
However, all the aforementioned works focus on optimization problems without constraints.
The work~\cite{PS-JC:21-csl} proposes a framework to employ Nesterov's accelerated algorithm 
for equality-constrained convex optimization problems.

In this paper, we extend the results of  HT in \cite{Gaudio2020AcceleratedRegressors} and~\cite{Jose2021-LCSS} for optimization in the presence of both equality and inequality constraints. 
For equality-constrained optimization problems, we show both for convex as well as nonconvex settings, 
that HT-based algorithms can be derived and shown to lead to a stable behavior while guaranteeing the feasibility at all times.
Our solution strategy is based on reformulating the original problem as the unconstrained optimization of a loss function proposed in~\cite{PLD-DR-JZK:21}
and identifying the conditions under which the loss function is convex as well as strongly convex.
We show how these algorithms can be extended to the case when inequality constraints are present as well. 
Conditions under which stable behavior can be guaranteed are clearly delineated in all cases. 

The organization of the paper is as follows. 
Section~\ref{sec:problem} outlines the problem statement and the approach we adopt to find its solution,
along with a few preliminaries on convex analysis and constrained optimization.
In Section~\ref{sec:convex}, we first propose a HT-based algorithm for equality-constrained convex optimization problems 
and based on that, generalize our treatment to general convex problems involving inequality constraints.
We then extend our approach for convex optimization to a class of nonconvex problems in Section~\ref{sec:nonlinear}. 
In Section~\ref{sec:conclusions}, we summarize the main contributions of the paper and outline a few future research directions.

\section{Problem Statement}\label{sec:problem}
We consider optimization problems of the form 
\begin{equation}\label{eq:problem}
    \begin{aligned}
        \min \ &f(x) \\
        \text{s.t.} \enspace & h(x) = 0 \\
        & g(x) \le 0,
    \end{aligned}
\end{equation}
where $x \in \real^n$ is the decision variable,
$f:\real^n \to \real$, $h : \real^n \to \real^m$, and $g:\real^n \to \real^p$ are continuously differentiable functions.
Without loss of generality, we assume that problem~\eqref{eq:problem} is not overdetermined, i.e.,  $m \le n$.

Our aim is to design recursive algorithms that solve~\eqref{eq:problem}, exhibiting accelerated convergence as proposed by Nesterov for unconstrained convex optimization, cf.~\cite{YEN:83}. Our approach is based on reformulating the original problem as the unconstrained optimization of a loss function and then using high-order tuners for accelerated convergence, cf.~\cite{Jose2021-LCSS}.
To satisfy constraints closely during the evolution, we employ a constraint-completion and correction procedure described below.
As will become apparent, our approach will rely on ensuring that an underlying loss function is convex. We first show that this is indeed the case if the optimization problem is convex 
in Section~\ref{sec:convex}. We then generalize our treatment to a class of nonconvex optimization problems
in Section~\ref{sec:nonlinear}.

Before proceeding to the technical content of the paper,
we present our notational conventions and review basic concepts from convex analysis and constrained optimization below.

\subsection*{Notation}
Let $\real$ denote the set of real numbers. $\|.\|$ denotes the 2-norm of a vector or matrix. For a continuously differentiable function $f: \real^n \to \real$, $\nabla f$ denotes its gradient. 
$(.)^T$ denotes the transpose of a vector or matrix.
For vectors $x,y \in \real^n$, $x \ge y$ implies that the inequality holds elementwise.
For a vector $x \in \real^n$, with $1\le i<j \le n$, $x_{i:j}$ denotes the subvector with elements from the $i$-th entry of $x$ to the $j$-th entry.

\subsection*{Convex analysis}
Here we present the basics of convex analysis following~\cite{RTR:70,Nesterov2018LecturesOptimization}.
A set $\Omega \in \real^n$ is convex if for all $x, y \in \Omega$, $\lambda x + (1-\lambda) y \in \Omega$ for all $\lambda \in (0,1)$.
\begin{definition}\longthmtitle{Convex functions}
A  function $f:\real^n \to \real$ is convex on a convex set $\Omega$ if for all $x,y \in \Omega$,
\begin{equation*}
    f(\lambda x + (1-\lambda)y)\leq \lambda f(x)+ (1-\lambda) f(y),
\end{equation*}
for all $\lambda \in (0,1)$.
\end{definition}
\begin{definition}\longthmtitle{Strongly convex functions}
A continuously differentiable function $f$ is $\mu$-strongly convex on $\Omega$ if there exists a $\mu >0$ such that for all $x, y\in\Omega$,
\begin{equation*}
    f(\lambda x+ (1-\lambda)y)\le \lambda f(x)+ (1-\lambda) f(y) - \dfrac{1}{2}\mu \lambda (1-\lambda)\|x-y\|^2,  	
\end{equation*}
for all $\lambda \in (0,1)$.
\end{definition}
\begin{definition}\longthmtitle{$\overline{L}$-smooth convex functions}
A continuously differentiable function $f$ is $\overline{L}$-smooth convex on $\Omega$ if it is convex and there exists an upper bound $\overline{L} >0$
on the Lipschitz constant of its gradient, i.e., 
\begin{equation*}
   \dfrac{\|\nabla f(x) - \nabla f(y)\|}{\|x-y\|} \le \overline{L},  
\end{equation*}
for all $x, y\in \Omega$.
\end{definition}
A similar definition could be stated for $\overline{L}$-smooth strongly convex functions
but we omit it here to avoid repetition.

\subsection*{Nonlinear optimization via constraint-completion and correction}
Consider the nonlinear optimization problem~\eqref{eq:problem}
and define a loss function $\Lc:\real^n \to \real$, consisting of the original objective function and soft loss terms penalizing the constraint violation as
\begin{equation*}
    \Lc(x) = f(x) + \lambda_h\|h(x)\|^2 + \lambda_g\|\softp(g(x))\|^2,
\end{equation*}
where $\lambda_h, \lambda_g > 0$ are design parameters. 
Given $y \in \real^p$, the function $\softp: \real^p \to \real^p$ is
defined as
\begin{equation*}
    \softp(y) = \log(1+e^{y}),
\end{equation*}
and serves as a smooth approximation to ReLU to ensure that  $\Lc$ is continuously differentiable. 
We employ a constraint-completion and correction approach to
leverage the fact that equality constraints introduce linear dependencies in the feasible solution space, as in~\cite{PLD-DR-JZK:21}. 
Building on this insight and assuming that problem~\eqref{eq:problem} is not overdetermined, 
$x$ is partitioned into an independent variable $\theta \in \real^m$ and a dependent variable $z \in \real^{n-m}$, 
\begin{equation*}
    x = [\theta^T \quad z^T]^T .
\end{equation*}  
We assume that $h$ is such that given $m$ entries of $x$, its remaining $(n-m)$ entries can be computed either in closed form or recursively.
In other words, we assume that we have knowledge of the function \(p:\mathbb{R}^m \rightarrow \mathbb{R}^{n-m}\)  such that 
\begin{equation*}
    h([\theta^T \quad p(\theta)^T]^T) = 0 ,
\end{equation*} 
holds for all \(\theta \in\mathbb{R}^m\). 
For all the points where $\frac{\partial h}{\partial z} \neq 0$,
existence and uniqueness of $p$ is guaranteed from the Implicit Function theorem.
%
The reduction of variable dimension as explained above ensures that equality constraints are always satisfied. 

Using the function $p(\cdot)$ defined as above, we now define a modified loss function $l:\mathbb{R}^m \rightarrow \mathbb{R}$ as
\begin{equation*}
    l(\theta) = \mathcal{L}([\theta^T \quad p(\theta)^T]^T).
\end{equation*}
The optimization problem in~\eqref{eq:problem} is now reformulated as
an unconstrained minimization problem given by
\begin{equation*}
    \min l(\theta),
\end{equation*}
with $\theta \in \mathbb{R}^m$ as the decision variable.
Depending on the information about the mapping $p$, 
gradient of the modified loss function could be computed either explicitly or using the Implicit Function theorem as in~\cite{BA-JZK:17}.

The above completion procedure takes care of the equality constraints. However, there is no guarantee associated with the satisfaction of inequality constraints. 
As will be shown later, we employ a gradient-based correction procedure that employs penalty terms using the $\softp$ function above and corrections that allow the solution to approach the feasible region along the manifold of points satisfying the equality constraints.

\section{Convex Optimization Problems}\label{sec:convex}
In this section, we consider convex optimization problems 
in the general form
\begin{equation}\label{eq:problem-convex}
    \begin{aligned}
        \min \ &f(x) \\
        \text{s.t.} \enspace &Ax = b \\
        & g(x) \le 0,
    \end{aligned}
\end{equation}
where $x \in \real^n$ is the decision variable,
$f:\real^n \to \real$ and $g:\real^n \to \real^p$ are continuously differentiable (strongly) convex functions, $A \in \real^{m \times n}$ and $b \in \real^m$. 
We start with problems involving just the equality constraints and then generalize our approach to problems involving both equality and inequality constraints.

\subsection{Equality-constrained convex optimization}\label{sec:equality}
Equality-constrained convex optimization problems have the general structure
\begin{equation}\label{eq:problem-equality}
    \begin{aligned}
        \min \ &f(x) \\
        \text{s.t.} \enspace & Ax = b. 
    \end{aligned}
\end{equation}
With $\lambda_h>0$, the loss function $\Lc$ and the modified loss function $l$ take the following forms:
\begin{subequations}\label{eq:losses-equality}
\begin{align}\label{eq:loss-equality}
 \Lc(x) =& f(x) + \lambda_h\|A x -b \|^2, \\
    l(\theta) =& f([\theta^T \; p(\theta)^T]^T).  \label{eq:mloss-equality}
\end{align}
\end{subequations}
Here we have used the fact that $\lambda_h\|A [\theta^T \; p(\theta)^T]^T -b \|^2=0$.
We note that $p$ is an affine function of $\theta$. Let
\begin{equation}\label{eq:p}
    p(\theta)=P\theta+q,
\end{equation}
where $P \in \real^{(n-m) \times m}$ and $q \in \real^{n-m}$.
$\Lc$ is convex by construction.
We characterize the convexity properties of $l$ in the following result.

\begin{proposition} \label{prop:linear}
\longthmtitle{Convexity of the modified loss function for equality-constrained convex programs}
For the equality-constrained convex optimization problem~\eqref{eq:problem-equality}, assume $f$ is a $\overline{L}$-smooth convex function, and let 
\begin{align}\label{eq:M}
    \overline{M}=\sqrt{1+\|P\|}\overline{L}.
\end{align}
Then $l$ is $\overline{M}$-smooth convex.
\end{proposition}
\begin{proof}
Convexity of $l$ follows in a straightforward manner from the definitions of $l$ and $p$ in~\eqref{eq:mloss-equality} and~\eqref{eq:p} respectively, and the convexity of $f$. 
For the $\overline{M}$-smoothness, using the chain rule, we have
\begin{align*}
    \dfrac{d l}{d \theta} =& \dfrac{\partial l}{\partial \theta} + \dfrac{\partial l}{\partial p(\theta)} \dfrac{\partial p(\theta)}{\partial \theta}.
    \end{align*}
    Hence
    \begin{align*}
    \nabla l = \nabla f_{1:m} + P^T \nabla f_{m+1:n}.
    \end{align*}
    Let us consider the gradient of $l$ at $\theta_1, \theta_2 \in \real^m$ and examine the Lipschitz constant of $l$.
    \begin{align*}
     \dfrac{   \|\nabla l(\theta_1) - \nabla l(\theta_2)\|}{\|\theta_1-\theta_2\|}
   \le &  \dfrac{\|\nabla f_{1:m} (\theta_1) - \nabla f_{1:m}(\theta_2)\|}{\|\theta_1-\theta_2\|} \\ 
   + &\dfrac{\|P^T (\nabla f_{m+1:n} (\theta_1) - \nabla f_{m+1:n} (\theta_2))\|}{\|\theta_1-\theta_2\|}
    \end{align*}
    Since $\| \nabla f_{i:j} \| \le \| \nabla f \|$ for all $i,j$, we have
    \begin{align}\label{eq:smooth-first}
       \dfrac{   \|\nabla l(\theta_1) - \nabla l(\theta_2)\|}{\|\theta_1-\theta_2\|} \le \dfrac{(1+\|P\|) \|\nabla f(x_1) - \nabla f(x_2)\|}{\|\theta_1-\theta_2\|}
    \end{align}
    To write the denominator of the above expression in terms of $x_1$ and $x_2$, remember that
    \begin{align*}
        \|\theta_1-\theta_2\|^2 + \|P\theta_1 - P \theta_2 \|^2 = \|x_1-x_2 \|^2.
    \end{align*}
    Using properties of the norm,
    \begin{align}
        \|\theta_1-\theta_2\|^2 + \|P\| \|\theta_1 - P \theta_2 \|^2 \ge & \|x_1-x_2 \|^2  \notag \\
        \|\theta_1-\theta_2\| \ge \dfrac{\|x_1-x_2\|}{\sqrt{1+\|P\|}}. \label{eq:theta-x-norm}
    \end{align}
    Using~\eqref{eq:smooth-first} and~\eqref{eq:theta-x-norm}, we have
    \begin{align*}
        \dfrac{\|\nabla l(\theta_1) - \nabla l(\theta_2)\|}{\|\theta_1-\theta_2\|} \le & \dfrac{\sqrt{1+\|P\|} \|\nabla f(x_1) - \nabla f(x_2)\|}{\|x_1-x_2\|} \\
     \le & \sqrt{1+\|P\|} \overline{L},
    \end{align*}
    where the last inequality follows from the $\overline{L}$-smoothness property of $f$. 
\end{proof}
The next result extends Proposition~\ref{prop:linear} to the case of strongly convex functions.
\begin{corollary}
\label{co:linear}
\longthmtitle{Strong convexity of the modified loss function for equality-constrained convex programs}
For the equality-constrained convex optimization problem~\eqref{eq:problem-equality}, assume $f$ is a $\overline{L}$-smooth and $\mu$-strongly convex function. 
Then $l$ is $\overline{M}$-smooth and $\mu$-strongly convex, where $\overline{M}$ is defined in equation~\eqref{eq:M}.
\end{corollary}
\begin{proof}
Consider $\theta_1,\theta_2 \in \real^m$.
Now consider $x_1,x_2 \in \real^n$ defined as
$x_1=[\theta_1^T \;\; p(\theta_1)^T]^T$ and $x_1=[\theta_2^T \;\; p(\theta_2)^T]^T$.
Then from the properties of $f$, it follows that
\begin{align*}
    f(\lambda x_1+ (1-\lambda)x_2)\le & \lambda f(x_1)+ (1-\lambda) f(x_2) \\ &- \dfrac{1}{2}\mu \lambda (1-\lambda)\|x_1-x_2\|^2,  	
\end{align*}
for all $\lambda \in (0,1)$. Using the fact that $\|\theta_1-\theta_2\|^2 \le \|x_1-x_2\|^2$, we have
\begin{align*}
    f(\lambda x_1+ (1-\lambda)x_2)\le & \lambda f(x_1)+ (1-\lambda) f(x_2) \\ &- \dfrac{1}{2}\mu \lambda (1-\lambda)\|\theta_1-\theta_2\|^2.	
\end{align*}
Rest of the proof follows from the definition of $l$ and the proof of Proposition~\ref{prop:linear}. 
\end{proof}

Now that we have established the convexity and smoothness properties of the modified loss function $l$, 
we leverage the properties of high-order tuners~\cite{Gaudio2020AcceleratedRegressors,Jose2021-LCSS} to propose an accelerated algorithm to solve~\eqref{eq:problem-equality}.  
Let $\N_k$ be the normalizing signal defined as
\begin{equation*}
    \N_k = 1 + H_k,
\end{equation*} 
where
\begin{equation*}
    H_k = \max\{\zeta: \zeta \in \sigma(\nabla^2 l(\theta_k))\},
\end{equation*}
where $\sigma(\nabla^2 l(\theta_k))$ denotes the spectrum of the Hessian matrix of the loss function $l$ evaluated at $\theta=\theta_k$.
Note that it is also possible to make a more conservative selection for $\N_k$ such as $\overline{M}$, i.e., smoothness parameter of the loss function if accurate information about $\nabla^2 l$ is not available.
Next we introduce Algorithm 1 to solve problem~\eqref{eq:problem-equality}.
\begin{algorithm}
\caption{HT Optimizer for equality-constrained convex optimization}
\begin{algorithmic}[1] 
\STATE \textbf{Initial conditions} $\theta_0$, $\nu_0$, gains $\gamma$, $\beta$
\FOR{$k=0$ to $N$}
\STATE Compute $\nabla l(\theta)$ and let \(\mathcal{N}_k=1+H_k\)
\STATE $\nabla \overline{q}_k(\theta_k)= \dfrac{\nabla l(\theta_k)}{\mathcal{N}_k}$
\STATE $\overline{\theta}_k=\theta_k-\gamma\beta\nabla \overline{q}_k(\theta_k)$
\STATE $\theta_{k+1}\leftarrow \overline{\theta}_k-\beta(\overline{\theta}_k-\nu_k) $
\STATE Compute $\nabla l(\theta_{k+1})$ and let
\STATE $\nabla \overline{q}_k(\theta_{k+1})= \dfrac{\nabla l(\theta_{k+1})}{\mathcal{N}_k}$
\STATE $\nu_{k+1} \leftarrow \nu_k -\gamma\nabla \overline{q}_k(\theta_{k+1})$
\ENDFOR
\end{algorithmic}
\end{algorithm}

The following result formally characterizes the convergence properties of Algorithm 1.
\begin{theorem}\label{thm:equality-cvx}
\longthmtitle{Convergence of the HT algorithm for equality-constrained convex programs}
If the objective function $f$ is $\overline{L}$-smooth convex, then  
with \(0<\beta<1\) and \(0<\gamma<\frac{\beta(2-\beta)}{8+\beta}\),  
the sequence of iterates $\{\theta_k\}$ generated by Algorithm 1 satisfy $\underset{k\rightarrow\infty}{\lim}l(\theta_k) = l(\theta^*)$, 
where $l(\theta^*)=f([\theta^{*T} \quad p(\theta^*)^T]^T)$ is the optimal value of~\eqref{eq:problem-equality}. 
\end{theorem}
\begin{proof}
From Proposition~\ref{prop:linear}, $\overline{L}$-smoothness of the objective function $f$ implies that the loss function $l$
is $\overline{M}$-smooth and convex.
Rest of the proof follows from~\cite[Theorem 2]{Jose2021-LCSS}.
\end{proof}
Following Corollary~\ref{co:linear} and~\cite[Theorem~3]{Jose2021-LCSS}, a similar result exists for the strongly convex case as well.

\subsection{Convex optimization problems with equality \& inequality constraints}
Here we extend our approach to solve general convex optimization problems involving equality as well as the inequality constraints in the form~\eqref{eq:problem-convex}.
Redefine the loss function and the modified loss functions~\eqref{eq:losses-equality} 
by including a penalty term corresponding to the inequality constraints violation as
\begin{subequations}
\begin{align}
    \mathcal{L}(x) =& f(x) + \lambda_h\|Ax-b\|^2 + \lambda_g \|\softp(g(x))\|^2, \\
    l(\theta) =& f([\theta^T \; p(\theta)^T]^T) + \lambda_g \|\softp(g([\theta^T \; p(\theta)^T])^T)\|^2. \label{eq:mloss-inequality}
\end{align}
\end{subequations}
Since the penalty term corresponding to the inequality constraints is convex,
it follows from Proposition~\ref{prop:linear} that the modified loss function~\eqref{eq:mloss-inequality} is convex.

Our approach is based on the inequality correction procedure of~\cite{PLD-DR-JZK:21}, briefly described in Section~\ref{sec:problem}. 
The method involves first implementing the HT Algorithm 1 on the loss function~\eqref{eq:mloss-inequality}
ensuring that the equality constraints are met at all times.
Then we apply an additional update that
drives the decision variable towards the feasible region corresponding to the inequality constraints as well.
Let $\alpha>0$ be the stepsize and define $\rho: \real^n \to \real^n$ as
\begin{align}\label{eq:rho}
\rho \Bigg(
  \begin{bmatrix}
           \theta\\
           p(\theta) \\
  \end{bmatrix}
\Bigg)
=
  \begin{bmatrix}
           \theta-\alpha\Delta\theta \\
           p(\theta)-\alpha\Delta p(\theta) \\
  \end{bmatrix},
\end{align}
where
\begin{equation*}
\begin{aligned}
    \Delta \theta = \Bigg( \dfrac{d}{d \theta} \| \softp \bigg(g\bigg(
\begin{bmatrix}
           \theta \\
           p(\theta) \\
 \end{bmatrix}
\bigg) \bigg)\|^2 \Bigg)^T
\end{aligned}.
\end{equation*}
Note immediately that the inequality correction step above does not affect the feasibility with respect to the equality constraints. 
Hence by implementing the described method, we obtain Algorithm 2 that satisfies equality constraints at each step and moves closer towards satisfying the inequality constraints with each successive iteration.
\begin{algorithm}
    \caption{HT optimizer for equality + inequality constrained convex optimization}
    \begin{algorithmic}[1] 
    \STATE \textbf{Initial conditions} $x_0$, $\overline{x}_0$, $\nu_0$, gains $\alpha$, $\gamma$, $\beta$
    \FOR{k=0,1,2,...}
    \STATE Compute $\nabla l(\theta_k)$ and let \(\mathcal{N}_k=1+H_k\)
    \STATE $\nabla \overline{f}_k(\theta_k)= \dfrac{\nabla l(\theta_k)}{\mathcal{N}_k}$
    \STATE $\overline{\theta}_k=\theta_k-\gamma\beta\nabla \overline{f}_k(\theta_k)$
    \STATE ${\theta}_{k+1}\leftarrow \overline{\theta}_k-\beta(\overline{\theta}_k-\nu_k) $
    \STATE $\overline{x}_{k+1} = [\theta^T_{k+1} \;\;p(\theta_{k+1})^T]^T$
    \STATE Compute $x_{k+1} \leftarrow \rho(\overline{x}_{k+1})$
    \STATE Compute $\nabla l(\theta_{k+1})$ and let
    \STATE $\nabla \overline{f}_k
    (\theta_{k+1})= \dfrac{\nabla l(\theta_{k+1})}{\mathcal{N}_k}$
    \STATE $\nu_{k+1} \leftarrow \nu_k -\gamma\nabla \overline{q}_k(\theta_{k+1})$
    \ENDFOR
    \end{algorithmic}
    \end{algorithm}

It is reasonable to expect that if the hypotheses of Theorem~\ref{thm:equality-cvx} are satisfied and
$\alpha$ is properly selected, Algorithm~2 solves problem~\eqref{eq:problem-convex}.

\section{Convex Optimization for a Class of Nonconvex Problems}\label{sec:nonlinear}
In this section, we extend our approach of achieving accelerated convergence via high-order tuners to a class of nonconvex optimization problems.
As with the convex case, we start with problems involving just the equality constraints and then generalize our methodology to problems 
involving inequality constraints as well.
Throughout this section, we consider only the conditions under which the loss function is convex. 
The arguments could be generalized easily to the strongly convex case.

\subsection{Equality-constrained nonconvex problems}
Consider the optimization problem
\begin{equation}\label{eq:problem-equality-nl}
    \begin{aligned}
        \min \ &f(x) \\
        \text{s.t.} \enspace & h(x) = 0,
    \end{aligned}
\end{equation}
and the associated loss function
\begin{align}\label{eq:loss-equality-nl}
    \Lc(x)=f(x)+\lambda_h \|h(x)\|^2.
\end{align}
The following result whose proof is immediate and hence skipped, provides the conditions under which $\Lc$ is convex.
\begin{lemma}\label{lemma:cvx-nl}
\longthmtitle{Convexity of the loss function for equality-constrained nonconvex programs}
If $f$ and $h$ are convex, then $\Lc$ defined in~\eqref{eq:loss-equality-nl} is convex. 
\end{lemma}
The modified loss function in this case once again takes the form~\eqref{eq:mloss-equality},
albeit the functional form of $p$ would not be linear anymore.
As such, establishing the convexity of $l$ over the entire domain as in Proposition~\ref{prop:linear} may not be feasible anymore.
We therefore search for conditions under which the modified loss function $l$ is convex over some subset of the domain.
We summarize a set of such conditions in the following result.

\begin{proposition}\label{prop:cvx-nl}
\longthmtitle{Convexity of the modified loss function for equality-constrained nonconvex programs}
Assume that there exists a convex set $\Omega_n \in \real^n$ such that
the functions $f$ and $h$ are convex on $\Omega_n$.
Let 
\begin{align}\label{eq:omega-m}
    \Omega_m=\{\theta \: | \: \theta=x_{1:m}, x \in \Omega_n \}.
\end{align}
If either of the following conditions is satisfied:
\begin{enumerate}
    \item $\nabla \Lc(x) \ge 0$ for all $x \in \Omega_n$, and $p$ is convex on $\Omega_m$,
    \item $\nabla \Lc(x) \le 0$ for all $x \in \Omega_n$, and $p$ is concave on $\Omega_m$,
    \end{enumerate}
    then $l$ is convex on $\Omega_m$.
\end{proposition}
\begin{proof} 
We present here the arguments for only condition~(i). The ensuing treatment easily generalizes to condition~(ii).
It is immediate to see that $\Omega_m$ is convex. 
Consider $\theta_1, \theta_2 \in \Omega_m$. 
Since $p(\theta)$ is convex, we have
\begin{equation*}
    p(\lambda \theta_1 + (1-\lambda)\theta_2) \leq \lambda p(\theta_1) + (1-\lambda)p(\theta_2),
\end{equation*}
and it follows that
\begin{equation*}
    \begin{bmatrix}
      \lambda \theta_1 + (1-\lambda)\theta_2   \\
        p(\lambda \theta_1 + (1-\lambda)\theta_2)
    \end{bmatrix}
    \leq 
    \lambda 
    \begin{bmatrix}
        \theta_1 \\
        p(\theta_1)
    \end{bmatrix} + \
    (1-\lambda)
    \begin{bmatrix}
        \theta_2 \\
        p(\theta_2)
    \end{bmatrix}.
\end{equation*}
Since $\mathcal{L}$ is nondecreasing, it follows that
\begin{align} \notag
    \mathcal{L}& \Bigg( \begin{bmatrix}
      \lambda \theta_1 + (1-\lambda)\theta_2   \\
        p(\lambda \theta_1 + (1-\lambda)\theta_2)
    \end{bmatrix} \Bigg) \\
    \leq &
    \mathcal{L} \Bigg(
    \lambda 
    \begin{bmatrix}
        \theta_1 \\
        p(\theta_1)
    \end{bmatrix} + \
    (1-\lambda)
    \begin{bmatrix}
        \theta_2 \\
        p(\theta_2)
    \end{bmatrix} \Bigg). \label{eq:L-nondecreasing}
\end{align}
Moreover, it follows from Lemma~\ref{lemma:cvx-nl} that $\Lc$ is convex, and we have
\begin{align}\notag
    \mathcal{L}& \Bigg(
    \lambda 
    \begin{bmatrix}
        \theta_1 \\
        p(\theta_1)
    \end{bmatrix} + \
    (1-\lambda) 
    \begin{bmatrix}
        \theta_2 \\
        p(\theta_2)
    \end{bmatrix} \Bigg) \\
    \leq &
    \lambda \mathcal{L}\Bigg(
    \begin{bmatrix}
        \theta_1 \\
        p(\theta_1)
    \end{bmatrix}\Bigg) + 
    (1-\lambda) \mathcal{L}\Bigg(
    \begin{bmatrix}
        \theta_2 \\
        p(\theta_2)
    \end{bmatrix}\Bigg). \label{eq:L-cvx}
\end{align} 
Combining the inequalities~\eqref{eq:L-nondecreasing} and~\eqref{eq:L-cvx}, we get 
\begin{align*}
    \mathcal{L}&\Bigg(
    \begin{matrix}
        \lambda \theta_1 + (1-\lambda) \theta_2 \\
        p(\lambda \theta_1 + (1-\lambda) \theta_2)
    \end{matrix}\Bigg) \\
    \leq &
    \lambda \mathcal{L}\Bigg(
    \begin{bmatrix}
        \theta_1 \\
        p(\theta_1)
    \end{bmatrix}\Bigg) + 
    (1-\lambda) \mathcal{L}\Bigg(
    \begin{bmatrix}
        \theta_2 \\
        p(\theta_2)
    \end{bmatrix}\Bigg).
\end{align*}
And from the definition of the modified loss function, it follows that
\begin{equation*}
    l(\lambda \theta_1 + (1-\lambda)\theta_2) \leq \lambda l(\theta_1) + (1-\lambda) l(\theta_2),
\end{equation*}
completing the proof.
\end{proof}

Now that we have established sufficient conditions for the convexity of the modified loss function, we can use high-order tuners to find an optimizer of~\eqref{eq:problem-equality-nl}.
In fact, if the sequence of iterates lie within the set $\Omega_n$, 
then we can use Algorithm 1, stated earlier for convex programs, to find a solution of~\eqref{eq:problem-equality-nl}. The following result formalizes this.

\begin{theorem}\label{thm:nl}
\longthmtitle{Convergence of the HT algorithm for equality-constrained nonconvex programs}
If the objective function $f$ and the equality constraint $h$ are convex over a set $\Omega_n$,
In addition, with \(0<\beta<1\), \(0<\gamma<\frac{\beta(2-\beta)}{8+\beta}\), and $\theta_0 \in \Omega_m$, 
where $\Omega_m$ is defined in~\eqref{eq:omega-m},
if the sequence of iterates $\{\theta_k\}$ generated by Algorithm 1 satisfy $\{\theta_k\} \in \Omega_m$,
then $\underset{k\rightarrow\infty}{\lim}l(\theta_k) = l(\theta^*)$, 
where $l(\theta^*)=f([\theta^{*T} \quad p(\theta^*)^T]^T)$ is the optimal value of~\eqref{eq:problem-equality-nl}. 
\end{theorem}
\begin{proof}
Convexity of the loss function $l$ follows from Proposition~\ref{prop:cvx-nl}.
Moreover, since $\{\theta_k\} \in \Omega_m$ for all $k$, there exists a constant $\overline{S}$ such that
\begin{align*}
    \dfrac{\|\nabla l(\theta_1) - \nabla l (\theta_2)\|}{\|\theta_1 - \theta_2\|} \le \overline{S},
\end{align*}
for all $\theta_1, \theta_2 \in \Omega_m$.
Rest of the proof follows from~\cite[Theorem 2]{Jose2021-LCSS}.
\end{proof}
Theorem~\ref{thm:nl} enables us to leverage Algorithm 1, provided that the state remains inside the set over which the modified loss function is convex.
It is reasonable to argue that this is always not the case.
To overcome this assumption, we use the projection operator defined as
\begin{equation*}
    \proj(\tilde{\theta}) = \text{argmin} \| \tilde{\theta} - \theta \|, \quad \forall \theta \in \Omega_m
\end{equation*}
to make sure that the state remains inside the set $\Omega_m$.
Algorithm~3 states this concisely.
\begin{algorithm}
\caption{HT Optimizer for equality-constrained nonconvex optimization}
\begin{algorithmic}[1] 
\STATE \textbf{Initial conditions} $\theta_0$, $\nu_0$, gains $\gamma$, $\beta$
\STATE $\theta_0 \leftarrow \proj(\theta_0)$
\FOR{$k=1$ to $N$}
\STATE Compute $\nabla l(\theta)$ and let \(\mathcal{N}_k=1+H_k\)
\STATE $\nabla \overline{q}_k(\theta_k)= \dfrac{\nabla l(\theta_k)}{\mathcal{N}_k}$
\STATE $\overline{\theta}_k=\theta_k-\gamma\beta\nabla \overline{q}_k(\theta_k)$
\STATE $\theta_{k+1}\leftarrow \proj (\overline{\theta}_k-\beta(\overline{\theta}_k-\nu_k)) $
\STATE Compute $\nabla l(\theta_{k+1})$ and let
\STATE $\nabla \overline{q}_k(\theta_{k+1})= \dfrac{\nabla l(\theta_{k+1})}{\mathcal{N}_k}$
\STATE $\nu_{k+1} \leftarrow \nu_k -\gamma\nabla \overline{q}_k(\theta_{k+1})$
\ENDFOR
\end{algorithmic}
\end{algorithm}

The arguments of this section show how the proposed approach could be 
applied to solve nonconvex problems,
where a convex objective function needs to be optimized with respect to
nonlinear convex equality constraints.
Moreover, the proposed method ensures that the state remains feasible throughout the
evolution, regardless of the convexity of the original problem.

\subsection{Nonconvex optimization problems with equality \& inequality constraints}
Here we employ the correction procedure described in Section~\ref{sec:problem} to extend our approach to solve problems of the form~\eqref{eq:problem},
which includes inequality constraints in addition to equality constraints.
As in the convex case, let us define the function $\rho$ as in~\eqref{eq:rho} and consider the following algorithm,
obtained from appending Algorithm~3 with a suitable inequality correction step.
\begin{algorithm}
    \caption{HT optimizer for equality + inequality constrained nonconvex optimization}
    \begin{algorithmic}[1] 
    \STATE \textbf{Initial conditions} $x_0$, $\overline{x}_0$, $\nu_0$, gains $\alpha$, $\gamma$, $\beta$
    \STATE $x_0 \leftarrow \projx(x_0)$
    \FOR{k=0,1,2,...}
    \STATE Compute $\nabla l(\theta_k)$ and let \(\mathcal{N}_k=1+H_k\)
    \STATE $\nabla \overline{f}_k(\theta_k)= \dfrac{\nabla l(\theta_k)}{\mathcal{N}_k}$
    \STATE $\overline{\theta}_k=\theta_k-\gamma\beta\nabla \overline{f}_k(\theta_k)$
    \STATE ${\theta}_{k+1}\leftarrow \overline{\theta}_k-\beta(\overline{\theta}_k-\nu_k) $
    \STATE $\overline{x}_{k+1} = [\theta^T_{k+1} \;\;p(\theta_{k+1})^T]^T$
    \STATE Compute $x_{k+1} \leftarrow \projx(\rho(\overline{x}_{k+1}))$
    \STATE Compute $\nabla l(\theta_{k+1})$ and let
    \STATE $\nabla \overline{f}_k
    (\theta_{k+1})= \dfrac{\nabla l(\theta_{k+1})}{\mathcal{N}_k}$
    \STATE $\nu_{k+1} \leftarrow \nu_k -\gamma\nabla \overline{q}_k(\theta_{k+1})$
    \ENDFOR
    \end{algorithmic}
    \end{algorithm}

If the hypotheses of Theorem~\ref{thm:nl} are satisfied, then with an appropriate value of $\alpha$ introduced in~\eqref{eq:rho}, we can use Algorithm~4 to find a solution of~\eqref{eq:problem}, without the explicit assumption of states belonging to the set $\Omega_n$.

\section{Conclusions and Future Work}
\label{sec:conclusions}
We have presented accelerated algorithms based on high-order tuners for solving constrained convex optimization problems.
Our approach is based on identifying the conditions under which the reformulated loss function is convex, guarantees that the
equality constraints are satisfied at all times,
and is also applicable to a class of nonconvex optimization problems.
Future work will involve formally characterizing the rate of convergence and extending our approach to a broader class of nonconvex problems.

%
\bibliographystyle{IEEEtran}
\bibliography{alias.bib,main.bib}
\end{document}